\documentclass[a4paper,12pt]{article}

\usepackage[T2A]{fontenc}
\usepackage[cp1251]{inputenc}
\usepackage{amsfonts,amsmath,amsthm,amscd,amssymb,latexsym,amsbsy,pb-diagram,cite,caption}
\usepackage[mathscr]{eucal}

\usepackage{ifpdf}
\ifpdf
\usepackage[pdftex]{graphicx}
\usepackage[pdftex,dvipsnames,usenames]{color}
\else
\usepackage{graphicx}
\usepackage[dvips,dvipsnames,usenames]{color}
\fi

\newtheorem{definition}{Definition}[section]
\newtheorem{theorem}[definition]{Theorem}
\newtheorem{lemma}[definition]{Lemma}
\newtheorem{proposition}[definition]{Proposition}
\newtheorem{corollary}[definition]{Corollary}

\theoremstyle{definition}
\newtheorem{remark}[definition]{Remark}
\newtheorem{example}[definition]{Example}

\numberwithin{equation}{section}

\DeclareMathOperator{\diam}{diam}
\def\we{\mathrel{\stackrel{\rm w}=}}
\begin{document}

\begin{center}
{\Large\bf Weak similarities of metric and semimetric spaces}
\end{center}

\begin{center}
{\bf Oleksiy Dovgoshey and  Evgeniy Petrov}
\end{center}

\begin{abstract}
Let $(X,d_X)$ and $(Y,d_Y)$ be semimetric spaces with distance sets
$D(X)$ and, respectively, $D(Y)$. A mapping $F:X\to Y$ is a weak
similarity if it is surjective and there exists a strictly
increasing \mbox{$f:D(Y)\to D(X)$} such that $d_X=f\circ d_Y\circ
F$. It is shown that the weak similarities between geodesic spaces
are usual similarities and every weak similarity $F:X\to Y$ is an
isometry if $X$ and $Y$ are ultrametric and compact with
$D(X)=D(Y)$. Some conditions under which the weak similarities are
homeomorphisms or uniform equivalences are also found.
\end{abstract}

\bigskip

{\bf 2010 AMS Classification:} 54E40, 54E35, 54E25.

\section{Introduction}

In the paper we define the  notion of weak similarities of
semimetric spaces and study some properties of these mappings.
Before doing this work we remaind some definitions and introduce
related designations.

Let $X$ be a set. A \emph{semimetric} on $X$ is a function
$d:X\times X\to \mathbb{R}^+$, $\mathbb{R}^+=[0,\infty)$, such that
$d(x,y)=d(y,x)$ and $(d(x,y)=0)\Leftrightarrow (x=y)$ for all $x,y
\in X$. A pair $(X,d)$, where  $d$  is a semimetric on $X$, is
called a \emph{semimetric space} (see, for example, \cite[p.
7]{Bl}). A semimetric $d$ is a \emph{metric} if, in addition, the
\emph{triangle inequality} $d(x,y)\leqslant d(x,z)+d(z,y)$ holds for
all $x, y, z \in X$. A metric is an \emph{ultrametric} if we have
the \emph{ultrametric inequality} $d(x,y)\leqslant \max
\{d(x,z),d(z,y)\}$ instead of the triangle one. We shall denote by
$\textbf{SM}$, $\textbf{M}$ and $\textbf{UM}$  the classes of the
nonvoid semimetric spaces, the nonvoid metric spaces and,
respectively, the nonvoid ultrametric ones.

Let $(X,d_X)$ and $(Y,d_Y)$ be semimetric spaces. A mapping $\Phi:
X\to Y$ is a \emph{similarity} if $\Phi$ is bijective and there is a
positive number $r=r(\Phi)$, the \emph{ratio} of $\Phi$, such that
$$
d_Y(\Phi(x),\Phi(y))=rd_X(x,y)
$$
for all $x,y\in X$ (cf. \cite[p. 45]{Ed}). The \emph{isometries} are
similarities with the ratio $r=1$. The semimetric spaces $X$ and $Y$
are said to be \emph{isometric} if there exists an isometry $F:X\to
Y$.  We define the \emph{distance set} $D(X)$ of a nonvoid
semimetric space $(X,d)$ as
$$
D(X):=\{d(x,y): x,y \in X\}.
$$
The following concept  seems  to be  a natural generalization of
similarities of semimetric spaces.
\begin{definition}\label{def1.1}
Let  $(X,d_X)$, $(Y,d_Y)\in \textbf{SM}$. A surjective mapping
$\Phi:X\to Y$ is a weak similarity if there is a  strictly
increasing function $f:D(Y)\to D(X)$ such that the equality
\begin{equation}\label{eq1.4}
d_X(x,y)=(f\circ d_Y)(\Phi(x),\Phi(y))
\end{equation}
holds for all $x,y\in X$. Where $f\circ d_Y$ denotes the composition
of the functions $f$ and $d_Y$. Here a function $f$ is said to be a
\emph{scaling function} of $\Phi$.
\end{definition}

If $\Phi:X\to Y$ is a weak similarity, we write $X \we Y$, say that
$X$ and  $Y$  are \emph{weak equivalent} and that the pair
$(f,\Phi)$ is a \emph{realization} of $X\we Y$.

It is clear that every similarity is a weak similarity. Moreover a
weak similarity $\Phi:X\to Y$ with a scaling function $f:D(Y)\to
D(X)$ is a similarity with a ratio $r$ if and only if
\begin{equation}\label{eq1.10}
f(t)=\frac{1}{r}t
\end{equation}
for every $t\in D(Y)$ (see Lemma~\ref{lem3.2} below). It was shown
in ~\cite[Theorem 3.6]{DM} that if $f:\mathbb{R}^+\to\mathbb{R}^+$
is a bijection such that $f\circ d$ and $f^{-1}\circ d$ are metrics
for \textbf{every} metric $d$, then ~(\ref{eq1.10}) holds with some
$r>0$ for every $t\in \mathbb{R}^+$. In the present paper we have
found some conditions under which a weak similarity $\Phi:X\to Y$ is
a similarity or even an isometry for \textbf{given} $X$ and $Y$.

More precisely:
\begin{itemize}
  \item  Theorem~\ref{th3.1} shows that every weak similarity
$\Phi:X\to Y$ is an isometry if $X$ and $Y$ are ultrametric and
compact with $D(X)=D(Y)$;
  \item Corollary ~\ref{c9} of Theorem~\ref{c8} describes the general
structural conditions for the set $D(X)$ under which all weak
similarities $\Phi:X\to X$ are isometries;
  \item  In Theorem~\ref{th4.4} we
prove that a weak similarity $\Phi:X\to Y$ is a similarity if its
scaling function $f$ and the inverse $f^{-1}$ are subadditive in
some generalized sense;
  \item  Using Theorem~\ref{th4.4} we show that every
weak similarity $\Phi:X\to Y$ is a similarity if $X$ and $Y$ are
geodesic spaces (see Theorem~\ref{th4.7}).
\end{itemize}

Moreover, in Section 1 we study some common properties of weak
similarities and, in Section 2, find conditions under which weak
similarities are homeomorphisms  (see Proposition~\ref{prop2.5}) or
uniform equivalences (see Proposition~\ref{prop2.8*}).
\begin{proposition}\label{prop2.1}
The relation $\we$ is an equivalence on the class $\textbf{SM}$.
\end{proposition}
\begin{proof}
We must show that $\we$ is reflexive, symmetric and transitive.

\noindent \textbf{Reflexivity.} To prove the reflexivity it suffices
to take $X=Y$, $f(t)\equiv t$ and $\Phi(x)\equiv x$ in
~(\ref{eq1.4}).

\noindent \textbf{Symmetry.} Let $(X,d_X)$, $(Y,d_Y)\in \textbf{SM}$
and  let $X\we Y$ hold with a realization $(f,\Phi)$.
 Equality~(\ref{eq1.4}) implies the
inequality $(f\circ d_Y)(\Phi(x_1),\Phi(x_2))>0$ for every pair of
distinct $x_1, x_2 \in X$. Consequently $d_Y(\Phi(x_1),\Phi(x_2))>0$
because $f$ is strictly increasing and $f(0)=0$ (the last equality
is also follows from~(\ref{eq1.4})). Thus we have
\begin{equation}\label{eq1.7}
(x_1\neq x_2)\Rightarrow (\Phi(x_1)\neq \Phi(x_2)).
\end{equation}
The surjectivity of $\Phi$ and ~(\ref{eq1.7}) imply the existence of
the inverse mapping \mbox{$\Phi^{-1}:Y\to X$.} Note also that
~(\ref{eq1.4}) holds for all $x, y \in X$ if and only if
\begin{equation}\label{eq1.7*}
d_X=f\circ d_Y\circ (\Phi\otimes \Phi)
\end{equation}
where $\Phi \otimes\Phi$ is a mapping from $X\times X$ to $Y \times
Y$ satisfying $\Phi\otimes\Phi(x_1,x_2)=(\Phi(x_1),\Phi(x_2))$ for
every $(x_1,x_2)\in X\times X$. Since the left-hand side of
~(\ref{eq1.7*}) is surjective, the function $f$ is also surjective.
Consequently $f$ is bijective, so that there is the inverse function
\mbox{$f^{-1}:D(X)\to D(Y)$}. Note also that $\Phi^{-1}$ is
surjective and $f^{-1}$ is strictly increasing.
Rewriting~(\ref{eq1.7*}) in the form
$$
f^{-1}\circ d_X\circ (\Phi^{-1}\otimes\Phi^{-1})=d_Y
$$
we see that the relation $Y \we X$ holds with the realization
$(f^{-1},\Phi^{-1})$. Thus $\we$ is symmetric.

\noindent \textbf{Transitivity.} Suppose we have $X\we Y$ and $Y\we
Z$ with the corresponding realizations $(f,\Phi)$ and $(g,\Psi)$.
Since  $\Psi\circ\Phi$ is surjective and $f\circ g$ is strictly
increasing,  $X\we Z$ follows from the commutativity of the diagram
\begin{equation*}
\begin{diagram}
\node{X\times X} \arrow[2]{e,t}{\Phi\otimes \Phi} \arrow{s,l}{d_X}
\node[2]{Y\times Y} \arrow[2]{e,t}{\Psi\otimes \Psi} \arrow{s,l}{d_Y}
\node[2]{Z\times Z} \arrow{s,l}{d_Z}\\
\node{D(X)}\node[2]{D(Y)} \arrow[2]{w,t}{f}\node[2]{D(Z)}
\arrow[2]{w,t}{g}
\end{diagram}
\end{equation*}
where $\Psi\otimes\Psi(y_1,y_2)=(\Psi(y_1),\Psi(y_2))$ for all $y_1,
y_2\in Y$.
\end{proof}

\begin{corollary}\label{cor2.1}
If $\Phi:X\to Y$ is weak similarity with the scaling function
$f:D(Y)\to D(X)$ and, $\Psi:Y\to Z$ is a weak similarity with the
scaling function $g:D(Z)\to D(Y)$, then $\Psi\circ \Phi$ is a weak
similarity with the scaling function $f\circ g$.
\end{corollary}

The proof of Proposition~\ref{prop2.1} gives also the next
\begin{corollary}\label{cor2.2*}
If $(f,\Phi)$ is a realization of the equivalence $X\we Y$, then $f$
and $\Phi$ are bijective.
\end{corollary}
The following is closely related to Proposition 2.2 in~\cite{BDHM}.
\begin{proposition}\label{prop2.3}
Let $X\in \textbf{UM}$. Then the relation $X\we Y$ implies the
membership $Y\in \textbf{UM}$ for every $Y\in \textbf{SM}$.
\end{proposition}
We leave this proposition without any proof as an exercise to the
reader.
\begin{remark}\label{rem2.4}
Simple examples show that, in general, the membership
$X\in\textbf{M}$ and the relation $X\we Y$ do not imply $Y\in
\textbf{M}$. It can be shown that for an increasing function
$f:\mathbb{R}^+\to\mathbb{R}^+$ the function $f\circ d$ is a metric
for every $(X,d)\in \textbf{M}$ if and only if $f$ is subadditive,
$f(0)=0$, and $f(x)>0$ for every $x>0$ (see Theorem 4.1
in~\cite{DM1}).
\end{remark}

Let $d$ and $\rho$ be two semimetrics defined on the same set $X$.
Then $d$ and $\rho$ are said to be \emph{coincreasing} if the
equivalence
$$
(d(x,y)\leqslant d(z,w))\Leftrightarrow (\rho(x,y)\leqslant
\rho(z,w))
$$
holds for all $x, y, z, w \in X$ (cf. Definition 3.1 from
\cite{DM}). The following proposition is an analogy to Lemma 3.1
from ~\cite{DPK}.
\begin{proposition}\label{prop1.10}
Let $(X,d_X)$, $(Y,d_Y)\in \textbf{SM}$ and let
$X\xrightarrow{\Phi}Y$ be a bijection. The mapping
$(X,d_X)\xrightarrow{\Phi}(Y,d_Y)$ is a weak similarity if and only
if there is a semimetric $\rho_X$ on $X$ such that $\rho_X$ and
$d_X$ are coincreasing and $(X,\rho_X)\xrightarrow{\Phi}(Y,d_Y)$ is
an isometry.
\end{proposition}
The simple proof is omitted here. Proposition~\ref{prop1.10} shows,
in particular, that the weak similarities are closely connected with
the isotone degenerate metric products. See~\cite{DPK} for the exact
definitions and some results in this direction.

\section{Weak equivalence, homeomorphism and
\newline
uniform equivalence}

Now we present conditions under which weak equivalent metric spaces
are homeomorphic. For every $A\subseteq \mathbb{R}^+$ we shall
denote by $\mathrm{ac}A$ the set of all accumulation points of $A$
in the space $\mathbb{R}^+$ with the standard topology.
\begin{proposition}\label{prop2.5}
Let $(X,d_X)$ and $(Y,d_Y)$ belong to $\textbf{M}$ and let $X\we Y$.
Suppose that the equivalence
\begin{equation}\label{eq2.0}
(0\in \mathrm{ac}D(X))\Leftrightarrow (0\in \mathrm{ac}D(Y))
\end{equation}
holds. Then $X$ and $Y$ are homeomorphic.
\end{proposition}
\begin{proof}
Assume that  $0$ is an isolated point for both $D(X)$ and $D(Y)$.
Then $X$ and $Y$ are discrete as topological spaces. Let $(f,\Phi)$
be a realization of $X\we Y$. By Corollary~\ref{cor2.2*} the mapping
\mbox{$\Phi:X\to Y$} is a bijection. Every bijection between
discrete topological spaces is a homeomorphism. Thus $X$ and $Y$ are
homeomorphic in the case under consideration.

Consider now the case when $0\in \mathrm{ac}D(X)\cap
\mathrm{ac}D(Y)$. To prove that $X$ and $Y$ are homeomorphic it
suffices to show that the weak similarities $\Phi$ and $\Phi^{-1}$
are continuous. By definition $\Phi$ is continuous if the equality
\begin{equation}\label{eq2.1}
\lim\limits_{n\to\infty}d_Y(\Phi(x_0),\Phi(x_n))=0
\end{equation}
holds for every $x_0\in X$ and every sequence $\{x_n\}_{n\in
\mathbb{N}}$, $x_n\in X$, with $\lim\limits_{n\to
\infty}d_X(x_0,x_n)=0$. Equality~(\ref{eq2.1}) can be written in the
form
$$
\lim\limits_{n\to\infty}f^{-1}(d_X(x_0,x_n))=0.
$$
Hence to prove~(\ref{eq2.1}) it is sufficient  to show that the
scaling function \mbox{$f^{-1}:D(X)\to D(Y)$} is continuous at the
point $0$. The last is easy to see. Indeed, since $0\in
\mathrm{ac}D(Y)$, for every $\varepsilon>0$ there is a point
$p\in(0,\varepsilon)\cap D(Y)$. Since $f^{-1}$ is a bijection we can
find $r\in (0,\infty)\cap D(X)$ such that $f^{-1}(r)=p$. The
increase of $f^{-1}$ implies the inclusion $f^{-1}([0,r))\subseteq
[0,\varepsilon)$. Hence $f^{-1}$ is continuous at $0$. The
continuity of $\Phi$ follows. Similarly we can show that $\Phi^{-1}$
is continuous.
\end{proof}

It was shown in the previous proof that the scaling function
$f^{-1}$ is continuous at $0$ if $0 \in \mathrm{ac} D(Y)$. Since a
function defined on a subset of $\mathbb R$ is continuous if and
only if it is right and left continuous, we can obtain
\begin{proposition}\label{prop2.8*}
Let $(X,d_X)$ and $(Y,d_Y)$ belong to $\textbf{SM}$, let $f:D(Y)\to
D(X)$ be an increasing bijection and let $t\in D(Y)$. Then $f$
  is continuous at $t$ and  $f^{-1}$ is continuous at
  $f(t)$ if and only if
   $$(t\in \mathrm{ac}([t,\infty)\cap D(Y)))\Leftrightarrow
  (f(t)\in \mathrm{ac}([f(t),\infty)\cap D(X)))$$ and
  $$(t\in \mathrm{ac}([0,t]\cap D(Y)))\Leftrightarrow
  (f(t)\in \mathrm{ac}([0,f(t)]\cap D(X))).$$
\end{proposition}

Recall that a uniformly continuous mapping $F:X\to Y$ is a
\emph{uniform equivalence} if $F$ is bijective and the inverse
function $F^{-1}:Y\to X$ is also uniformly continuous~\cite[p.
2]{Is}.

Using Proposition~\ref{prop2.8*} and the symmetry of the relation
$\we$ we obtain
\begin{corollary}\label{cor2.8}
Let  $(X,d_X)$, $(Y,d_Y)\in \textbf{M}$ and let $\Phi:X\to Y$ be a
weak similarity. Then $\Phi$ is a uniform equivalence if and only if
~(\ref{eq2.0}) holds.
\end{corollary}

It is well-known that every uniform equivalence preserves the
completeness of metric spaces (see, for example, \cite[p. 171]{Se}).
Consequently Corollary~\ref{cor2.8} implies
\begin{proposition}\label{prop2.4}
Let $(X,d_X)$, $(Y,d_Y)\in \textbf{M}$, and $X\we Y$, and $0\in
\mathrm{ac}(D(X))\cap \mathrm{ac}(D(Y))$. Suppose that $(X,d_X)$ is
complete. Then $(Y,d_Y)$ is also complete.
\end{proposition}

The following example shows that there exist metric spaces which are
weak equivalent but not homeomorphic.
\begin{example}\label{ex2.6}
Let $\{r_n\}_{n=1}^{\infty}$ and $\{p_n\}_{n=1}^{\infty}$ be
strictly decreasing sequences of positive real numbers such that
$\lim\limits_{n\to\infty}r_n=0$ and
$\lim\limits_{n\to\infty}p_n=p>0$. Let $Y:=\{y_0, y_1,...,y_n,...\}$
and $X:=\{x_0,x_1,...,x_n,...\}$ be some families of pairwise
distinct points. Define semimetrics $d_X$ and $d_Y$ by the rules
\begin{equation}\label{eq2.2}
\begin{split}
d_X(x_i,x_j):=
\begin{cases}
0 &\text{if} \,  \ i=j\\
 r_{i\vee j} &\text{if} \, \ i\wedge j =0 \, \ \text{and}\, \ i\vee
 j>0\\
  r_{i\wedge j} &\text{if} \, \ i\wedge j >0 \, \ \text{and}\, \ i\neq
  j,
\end{cases}\\
d_Y(y_i,y_j):=
\begin{cases}
0 &\text{if} \,  \ i=j\\
 p_{i\vee j} &\text{if} \, \ i\wedge j =0 \, \ \text{and}\, \ i\vee
 j>0\\
  p_{i\wedge j} &\text{if} \, \ i\wedge j >0 \, \ \text{and}\, \ i\neq
  j
\end{cases}
\end{split}
\end{equation}
where $i\vee j=\max\{i,j\}$ and  $i\wedge j=\min\{i,j\}$.   It can
be proved directly that $(X,d_X)$, $(Y,d_Y) \in \textbf{UM}$. The
functions $\Phi$ and $f$ defined as
\begin{equation}\label{eq2.3}
f(0):=0,\quad f(p_i):=r_i, \quad \Phi(x_0):=y_0,\quad
\Phi(x_i):=y_i,\quad \, \ i=1,2,...
\end{equation}
 are bijective and, moreover, $f$ is increasing. It follows from~(\ref{eq2.2}) and~(\ref{eq2.3}) that
$$d_X(x_i,x_j)=f(d_Y(\Phi(x_i),\Phi(x_j))$$ for all $x_i, x_j \in
X$. Consequently we have $X \we Y$ with the realization $(f,\Phi)$.
It still remains to  note that $X$ and $Y$ are not homeomorphic,
because $X$ has the limit point $x_0$ but $Y$ is discrete.
\end{example}
In the next example we consider some ultrametric spaces $X$ and $Y$
such that:
\begin{itemize}
  \item $X$ and $Y$ are homeomorphic,
  \item $X\we Y$ with the realization $(f,\Phi)$ for which $0$ is not a
  point of continuity of the scaling function $f:D(Y)\to D(X)$.
\end{itemize}

\begin{example}\label{ex2.6*}
Let $\{r_n\}_{n=1}^{\infty}$ and $\{p_n\}_{n=1}^{\infty}$ be the
sequences from the previous example. Let
$X:=\{x_1^1,...,x_n^1,...\}\cup\{x_1^2,...,x_n^2,...\}$ and
$Y=\{y_1^1,...,y_n^1,...\}\cup\{y_1^2,...,y_n^2,...\}$ be some
families of pairwise distinct points. Define semimetrics $d_X$ and
$d_Y$ by the rules
\begin{equation}\label{eq2.4}
\begin{split}
d_X(x,y):=
\begin{cases}
0 &\text{if} \,  \ x=y\\
 p_{i} &\text{if} \, \ x=x_i^1 \text{ and } y=x_i^2
  \text{ or if } x=x_i^2 \text{ and }
 y=x_i^1\\
 p_1 &\text{otherwise},
\end{cases}\\
 d_Y(x,y):=
\begin{cases}
0 &\text{if} \,  \ x=y\\
 r_{i} &\text{if} \, \ x=y_i^1 \text{ and } y=y_i^2
  \text{ or if } x=y_i^2 \text{ and }
 y=y_i^1\\
 r_1 &\text{otherwise}.
\end{cases}
\end{split}
\end{equation}
It can be proved directly that $X$ and $Y$ are countable,
ultrametric and discrete. Consequently $X$ and $Y$ are homeomorphic.

Let $\Phi:X\to Y$ and $f:D(Y)\to D(X)$ be the functions such that
\begin{equation}\label{eq2.5}
f(r_i)=p_i \ \text{ and } \ \Phi(x_i^j)=y_i^j,\quad j=1,2,\quad
i=1,2,\dots.
\end{equation}
Then $\Phi$ and $f$ are bijective and $f$ is increasing.
Equalities~(\ref{eq2.4}) and ~(\ref{eq2.5}) imply
$d_X(x,y)=f(d_Y(\Phi(x),\Phi(y)))$ for all $x,y \in X$. Consequently
$X$ and $Y$ are weak equivalent. The point $0$ is not a point of
continuity of the function $f:D(Y)\to D(X)$ because $0\in
\mathrm{ac}D(Y)$ and $0\notin \mathrm{ac}D(X)$ and $f$ is strictly
increasing.
\end{example}
\begin{remark}\label{rem2.10}
Proposition ~\ref{prop2.5} and Corollary~\ref{cor2.8} can be proved
also when $(X,d_X)$, $(Y,d_Y)\in \textbf{SM}$ if we suppose that the
distance functions $d_X$ and $d_Y$ are continuous. See
\cite[p.9]{Bl} for some basic results related to semimetric spaces
with continuous distance functions.
\end{remark}

\section{Rigidity of distance sets, weak similarities and isometries}

It this section we have found some conditions under which the weak
similarities become isometries.
\begin{theorem}\label{th3.1}
Let $(X,d_X)\in\textbf{UM}$ and $(Y,d_Y)\in \textbf{SM}$. If $X\we
Y$, $D(X)=D(Y)$ and $X$ is compact, then $X$ and $Y$ are isometric.
\end{theorem}
The next simple lemma is an original point of our considerations.
\begin{lemma}\label{lem3.2}
Let $(X,d_X)$, $(Y,d_Y) \in \textbf{SM}$ and let $r>0$. If we have
$X\we Y$ with a realization $(f,\Phi)$, then the following
conditions are equivalent
\begin{itemize}
  \item [(i)] The mapping $\Phi:X\to Y$ is a similarity with a ratio $r$.
  \item [(ii)] The function $f:D(Y)\to D(X)$ satisfies the equality
  \begin{equation}\label{eq3.0}
    f(t)=\frac{1}{r}t
  \end{equation}
for every $t\in D(Y)$.
\end{itemize}
\end{lemma}
\begin{proof}
If $\Phi$ is a similarity with a ratio $r$, then we have
  \begin{equation}\label{eq3.1}
    d_Y(\Phi(x),\Phi(y))=rd_X(x,y)
  \end{equation}
for all $x,y\in X$. This equality and ~(\ref{eq1.4}) imply
$$d_X(x,y)=\frac{1}{r}d_Y(\Phi(x),\Phi(y))=f(d_Y(\Phi(x),\Phi(y)))$$
for all $x,y \in X$. Consequently ~(\ref{eq3.0}) holds for every
$t\in D(Y)$, so that (i)$\Rightarrow$(ii) follows.

Analogously ~(\ref{eq3.0}) implies~(\ref{eq3.1}). Since $\Phi$ is a
bijection, from ~(\ref{eq3.1}) follows that $\Phi$ is a similarity.
\end{proof}
\begin{corollary}\label{cor3.2*}
Let $(X,d_X)$, $(Y,d_y)\in \textbf{SM}$ and let $\Phi:X\to Y$ be a
weak similarity with a scaling function $f$. Then $\Phi$  is an
isometry if and only if $f(t)=t$ for every $t\in D(Y)$.
\end{corollary}

A partially ordered set $P$ is called \emph{rigid} if there is one
and only one order preserving bijection $F:P\to P$, (see \cite[p.
343]{Har}). Of course if $P$ is rigid, then the unique order
preserving bijection of $P$ is the identical mapping.

 Corollary~\ref{cor3.2*} implies the following
\begin{corollary}\label{cor3.3}
Let $X,Y\in\textbf{SM}$ and let $D:=D(X)=D(Y)$. If $D$ is rigid and
$X\we Y$ with a realization $(f,\Phi)$, then the weak similarity
$\Phi:X\to Y$ is an isometry.
\end{corollary}

To obtain conditions under which $D$ is rigid we recall some notions
from the theory of ordered sets.

A total-ordered set $(S,\leqslant)$ is \emph{well-ordered} if every
nonempty subset of $S$ has a least element. In this case the
relation $\leqslant$ is referred to as a \emph{well-ordering}.
Similarly a total order $\leqslant$ on a set $S$ is a \emph{converse
well-ordering} if every nonempty subset of $S$ has a greatest
element. In what follows we consider a subset $D$ of $\mathbb{R}^+$
together with the standard order $\leqslant_D$ induced from
$(\mathbb{R}^+,\leqslant)$.
\begin{lemma}\label{lemm3.5}
Let $(D,\leqslant_D)$ be a nonempty subset of $\mathbb R^+$. If
$\leqslant_D$ is a well-ordering or a converse well-ordering, then
$D$ is rigid.
\end{lemma}
\begin{proof}
It is well known that every order preserving mapping of a
well-ordered set onto itself  is the identity mapping (see, for
example,\cite[p. 4]{Har}). Hence $D$ is rigid if $\leqslant_D$ is a
well-ordering. Using the duality principle~\cite[p. 47]{Har} we
obtain that $D$ is also rigid when $\leqslant_D$ is a converse
well-ordering.
\end{proof}
A poset $P$ is said to satisfy the ascending chain condition (ACC)
if given arbitrary infinite sequence of elements of $P$
$$p_1\leqslant p_2\leqslant \dots,$$ then there is $n\in\mathbb N$
such that $p_n=p_{n+1}=p_{n+2}\dots$. It is known that a
total-ordered set  is a converse well-ordered set if and only if ACC
holds.
\begin{lemma}\label{lemm3.6}
Let $(X,d_X)$ be a compact nonvoid ultrametric space and let
$D:=D(X)$. Then the ordered set $(D,\leqslant_D)$ is a converse
well-ordered set.
\end{lemma}
\begin{proof}
Suppose that  ACC  does not hold for $(D, \leqslant)$. Then there is
an infinite strictly increasing sequence
\begin{equation}\label{eq1.12}
r_1<r_2<\dots<r_n<r_{n+1}<\dots
\end{equation}
with $r_n\in D$, $n=1,2,...$. Let us denote by  $x_n$ and $y_n$ the
points of $X$ such that $d_X(x_n,y_n)=r_n$, $n=1,2,\dots$.

Since $(X,d_X)$ is compact, there is a strictly increasing sequence
$\{n_k\}_{k\in \mathbb{N}}$ of positive integer numbers such that
the sequences $\{x_{n_k}\}_{k\in \mathbb N}$ and $\{y_{n_k}\}_{k\in
\mathbb N}$ are convergent. Write
\begin{equation}\label{eq1.13}
x^*:=\lim\limits_{k\to\infty}x_{n_k},\quad
y^*:=\lim\limits_{k\to\infty}y_{n_k}
\end{equation}
and $r^*:=d_X(x^*,y^*)$. Since the function $d_X:X\times X\to D$ is
continuous, we have $r^*=\lim\limits_{k\to\infty}r_{n_k}$.
Using~(\ref{eq1.12}) we see that $r^*>0$. Relations ~(\ref{eq1.13})
are equivalent to
$$
\lim\limits_{k\to\infty}d_X(x^*,x_{n_k})=\lim\limits_{k\to\infty}d_X(y^*,y_{n_k})=0.
$$
Consequently there is $k_0\in \mathbb N$ such that
\begin{equation}\label{eq1.14}
d_X(x^*,x_{n_k})<r^*\mbox{ and }d_X(y^*,y_{n_k})<r^*
\end{equation}
for every $k\geqslant k_0$. Considering the triangle ($x^*$, $y^*$,
$x_{n_k}$) and using the first inequality from~(\ref{eq1.14}), we
see that the ultrametric inequality implies
$d_X(x^*,y^*)=d_X(y^*,x_{n_k})$ (see Figure 1). Similarly, the last
equality and the second inequality from~(\ref{eq1.14}) imply
$d_X(y^*,x_{n_k})=d_X(x_{n_k},y_{n_k})$.
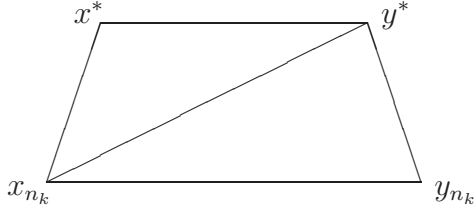
\begin{figure}[h]
\begin{center}
\begin{picture}(200,100)
\put(15,10){$x_{n_k}$} \put(175,10){$y_{n_k}$} \put(40,75){$x^*$}
\put(155,75){$y^*$} \put(30,15){\line(1,3){20}}
\put(30,15){\line(2,1){120}} \put(170,15){\line(-1,3){20}}
\put(30,15){\line(1,0){140}} \put(50,75){\line(1,0){100}}
\end{picture}
\end{center}
\caption{The quadruple $x^*, y^*, x_{n_k}, y_{n_k}$ contains two
distinct sides of the maximal length.}
\end{figure}
Consequently if $k>k_0$, then $d(x_{n_k},y_{n_k})=d(x^*,y^*)$,
contrary to ~(\ref{eq1.12}). Hence the poset $(D,\leqslant_D)$
satisfies ACC, i.e., $(D,\leqslant_D)$ is a converse well-ordered
set.
\end{proof}
\begin{proof}[Proof of Theorem~\ref{th3.1}]
Suppose that $X\we Y$, $D:=D(X)=D(Y)$ and $X$ is compact. We must
show that $X$ and $Y$ are isometric. Let $(f,\Phi)$ be a realization
of $X\we Y$. By Corollary~\ref{cor3.2*}, to prove that $X$ and $Y$
are isometric it suffices to show that $f$ is the identical
function. Lemma~\ref{lemm3.5} implies that $f:D\to D$ is identical
if $(D,\leqslant_D)$ is a converse well-ordered set. Since $X$ is
compact and ultrametric, it follows from Lemma~\ref{lemm3.6} that
$\leqslant_D$ is a converse well-ordering. Consequently $X$ and $Y$
are isometric.
\end{proof}
The above given proof also justifies the following
\begin{corollary}\label{cor3.7}
If $\Phi:X\to X$ is a weak similarity and $X$ is a compact
ultrametric space, then $\Phi$ is an isometry.
\end{corollary}
Let $(X,d_X)$, $(Y,d_Y)\in\textbf{SM}$. If $X\we Y$, then $D(X)$ and
$D(Y)$ must be ``isomorphic as ordered sets''. Let us recall some
related definitions.
\begin{definition}\label{c1} \cite[p. 45]{Har}
Let $(P_1,\leqslant)$ and $(P_2,\leqslant)$ be posets. Then
$(P_1,\leqslant)$ is said to be \emph{order-isomorphic}, if there
exists a bijection $f:P_1\to P_2$ with the property: For all $a,b\in
P_1$ there holds $a\leqslant b \Leftrightarrow f(a)\leqslant f(b)$.
\end{definition}
\begin{definition}\label{c2} \cite[p. 36]{Har}
The class of all posets which are order-isomorphic to a given poset
$(P,\leqslant)$ is called the \emph{order-type} of $(P,\leqslant)$
\end{definition}
The order-type of a poset $P$ will be defined as $\mathrm{tp} P$.
From definitions ~\ref{def1.1}, ~\ref{c1}, and ~\ref{c2} we obtain
\begin{proposition}\label{c3}
Let $(X,d_X)$, $(Y,d_Y)\in \textbf{SM}$. If $X\we Y$, then the
equality $\mathrm{tp}(D(X))=\mathrm{tp}(D(Y))$ holds.
\end{proposition}
\begin{corollary}\label{c4}
Let $(X,d_X)$, $(Y,d_Y)\in \textbf{SM}$ and let $X\we Y$. If $D(X)$
is rigid, then $D(Y)$ is also rigid.
\end{corollary}
\begin{proposition}\label{c5}
Let $(X,d_X)$, $(Y,d_Y)\in \textbf{SM}$ and let $\Phi_i:X\to Y$,
$i=1,2$, be weak similarities. If $D(X)$ is rigid, then there are
isometries $F:X\to X$ and $\Psi:Y\to Y$ such that
\begin{equation}\label{y1}
\Phi_2 = \Phi_1\circ F \ \mbox{  and  }\ \Phi_2=\Psi\circ \Phi_1.
\end{equation}
\end{proposition}
\begin{proof}
By Corollary~\ref{cor2.2*} the function $F:=\Phi^{-1}_1\circ \Phi_2$
is a weak similarity. Suppose that $D(X)$ is rigid. Then
Corollary~\ref{cor3.3} implies that $F$ is an isometry on $X$. Thus
$$
\Phi_2 = \Phi_1\circ F
$$
where $F:X\to X$ is an isometry. The first equality from
~(\ref{y1}) is proved. The second can be proved similarly.
\end{proof}
The structure of rigid total-ordered sets was described by A. C.
Morel in~\cite{Mo}. To apply his result in our studies we recall the
following
\begin{definition}\label{c6}
Let $(I,\prec)$ be a poset, and let $(P_i,\leqslant_i)$ be posets
for $i\in I$ with pairwise disjoint carrier sets $P_i$. Then we
define the \emph{ordered sum} $\sum\limits_{i\in
I}(P_i,\leqslant_i)$ of the posets $(P_i,\leqslant_i)$ over the
ordered argument $(I,\prec)$ as the poset $(V,\leqslant)$, where
$V:=\bigcup\limits_{i\in I}P_i$ and where $\leqslant$ is now defined
by: For $a,b \in V$ we put $a\leqslant b$ $\Leftrightarrow$ $a$ and
$b$ are in the same summand $P_i$, and there holds $a\leqslant_i b$,
or $a\in P_i$ and $b\in P_j$ with $i\neq j$ and $i\prec j$.
\end{definition}
\begin{definition}\label{c7}
Let $(I,\prec)$ be a poset, and let $\tau_i$, $i\in I$ be
order-types.  We take a poset $(P_i,\leqslant_i)$ with
$\mathrm{tp}(P_i)=\tau_i$ for every $i\in I$, and so that all $P_i$
are pairwise disjoint. Then the sum $\sum\limits_{i\in I}\tau_i$ is
the order-type of the ordered sum $\sum\limits_{i\in
I}(P_i,\leqslant_i)$. If we have $I=\{1,...,n\}$ with the standard
order $\leqslant$, then we set $
\tau_1+\cdots+\tau_n:=\mathrm{tp}\left(\sum\limits_{i\in
I}(P_i,\leqslant_i)\right). $
\end{definition}

In the case when all order-types $\tau_i$ are the same we define the
product
$$
\tau\mathrm{tp}(I):=\mathrm{tp}\left(\sum\limits_{i\in
I}(P_i,\leqslant_i)\right)
$$
where $\tau$ is the common order-type of the posets
$(P_i,\leqslant_i)$, $i\in I$.

\begin{lemma}[\cite{Mo} ]\label{c7*}
Let $A$ be a nonempty total-ordered set. The following statements
are equivalent.
\begin{itemize}
  \item [(i)] There is a nonidentical order preserving bijection $f:A\to
  A$.
  \item [(ii)] There are total-ordered sets $A_1$, $A_2$, $A_3$,
  $A_2\neq\varnothing$, such that
  $$
  \mathrm{tp}A=\mathrm{tp}A_1+\mathrm{tp}A_2\mathrm{tp}\mathbb
  Z+\mathrm{tp}A_3
  $$
\end{itemize}
where $\mathrm{tp}\mathbb Z$ is the order-type of the set $\mathbb
Z$ of all integer numbers with the standard order.
\end{lemma}

\begin{theorem}\label{c8}
Let $D_1$ and $D_2$ be subset of $\mathbb{R}^+$ such that $0\in
D_1\cap D_2$. Then the following statements are equivalent.
\begin{itemize}
  \item [(i)] There are sets $A_1$, $A_2$, $A_3\subseteq \mathbb R^+$, $A_2\neq
  \varnothing$,
  such that
  \begin{equation}\label{y2}
    \mathrm{tp}(D_1)=\mathrm{tp}(D_2)=\mathrm{tp}(A_1)+\mathrm{tp}(A_2)\mathrm{tp}(\mathbb
    Z)+\mathrm{tp}(A_3).
  \end{equation}
  \item [(ii)] There are $(X,d_X)$, $(Y,d_Y)\in \textbf{UM}$ and
  weak similarities $\Phi_1:X\to Y$, $\Phi_2:X\to Y$ such that
  $D(X)=D_1$ and $D(Y)=D_2$ and $\Phi_1^{-1}\circ \Phi_2$ is not an
  isometry.
  \item [(iii)] There are $(X,d_X)$, $(Y,d_Y)\in \textbf{SM}$ and
  weak similarities $\Phi_1: X\to Y$, $\Phi_2: X\to Y$ such that
  $D(X)=D_1$ and $D(Y)=D_2$ and $\Phi^{-1}_1\circ\Phi_2$ is not an
  isometry.
\end{itemize}
\end{theorem}
\begin{proof}
\textbf{(i)$\Rightarrow$(ii)}. Suppose that $A_1$, $A_2$ and $A_3$
are subsets of $\mathbb R^+$ for which ~(\ref{y2}) holds and
$A_2\neq \varnothing$. Then by Lemma~\ref{c7*} there exist strictly
increasing bijections $\Phi_i:D_1\to D_2$, $i=1,2$, such that
$\Phi_1\neq \Phi_2$. Write $X:=D_1$ and $Y:=D_2$ and define
semimetrics $d_X$ and $d_Y$ by the rules
  \begin{equation}\label{y3}
  \begin{split}
  d_X(x,y):=
  \begin{cases}
  0 &\text{if} \, \ x=y, \,\, x,y\in X\\
  \max\{x,y\} &\text{if} \, \ x \neq y, \,\, x,y \in X,
  \end{cases}\\
d_Y(x,y):=
  \begin{cases}
  0 &\text{if} \, \ x=y, \,\, x,y\in Y\\
  \max\{x,y\} &\text{if} \, \ x \neq y, \,\, x,y \in Y.
  \end{cases}
  \end{split}
  \end{equation}
Since $\max\{x,y\}\leqslant \max\{x, y,
z\}=\max\{\max\{x,z\},\max\{z,y\}\}$ for all $x,y,z\in\mathbb{R}^+$,
the semimetrics $d_X$ and $d_Y$ are ultrametrics. It is clear that
$D(X)=D_1$ and $D(Y)=D_2$. Write $f_i:=\Phi_i^{-1}$, $i=1,2$. Then
$D_2$ is the domain of $f_i$ and $D_1$ is the range of $f_i$ for
$i=1,2$. Since $\Phi_1$ and $\Phi_2$ are strictly increasing and
bijective, (\ref{y3}) implies
$$
f_i (d_Y
(\Phi_i(x_1),\Phi_i(X_2)))=f_i(\max\{\Phi_i(x_1),\Phi_i(x_2)\})
$$
$$
=\max\{f_i(\Phi_i(x_1)),f_i(\Phi_i(x_2))\}=\max\{x_1,x_2\}=d_X(x_1,x_2)
$$
for $i=1,2$ and all distinct $x_1, x_2\in X$. Moreover we have
$$
f_i (d_Y(\Phi_i(x),\Phi_i(x)))=f_i(0)=0=d_X(x,x)
$$
for $i=1,2$ and every $x\in X$. Consequently $\Phi_1$ and $\Phi_2$
are weak similarities. To prove that $\Phi_1^{-1}\circ\Phi_2$ is not
an isometry, it is sufficient to find $x, y\in X$ for which
  \begin{equation}\label{y4}
   d_X(x,y)\neq d_X(\Phi_1^{-1}(\Phi_2(x)),\Phi_1^{-1}(\Phi_2(y))).
  \end{equation}
Since the functions $\Phi_1$ and $\Phi_2$ are different and
$\Phi_1(0)=\Phi_2(0)=0$, there is $x_0\in D_1$ such that $x_0\neq 0$
and $\Phi_1(x_0)\neq\Phi_2(x_0)$ i.e.,
  \begin{equation}\label{y5}
    x_0\neq \Phi_1^{-1}(\Phi_2(x_0)).
  \end{equation}
Putting $x=x_0$ and $y=0$ and using~(\ref{y5}),~(\ref{y3}) we obtain
$$
d_X(x,y)=\max\{x_0,0\}=x_0\neq \Phi_1^{-1}(\Phi_2(x_0))
$$
$$
=\max\{\Phi_1^{-1}(\Phi_2(x_0)),\Phi_1^{-1}(\Phi_2(0))\}=
d_X(\Phi_1^{-1}(\Phi_2(x)),\Phi_1^{-1}(\Phi_2(y)))
$$
Relation~(\ref{y4}) follows.

\textbf{(ii)$\Rightarrow$(iii)}. This is trivial.

\textbf{(iii)$\Rightarrow$(i)}. Let (iii) hold. Proposition~\ref{c3}
implies the equality
  \begin{equation}\label{y6}
  \mathrm{tp}(D_1)=\mathrm{tp}(D_2).
  \end{equation}
We claim that the ordered sets $D_1$ and $D_2$ are not rigid. Indeed
from~(\ref{y6}) follows that $D_1$ is rigid if and only if $D_2$ is
rigid. By Proposition ~\ref{c5} if $D_1$ is rigid, then
$\Phi_1^{-1}\circ\Phi_2$ is an isometry (contrary to statement
(iii)). Consequently $D_1$ is not rigid. Lemma~\ref{c7*} implies
that there are total-ordered sets $A_1$, $A_2$, $A_3$ such that,
$A_2\neq \varnothing$ and
$$
\mathrm{tp}(D_1)=\mathrm{tp}(A_1)+\mathrm{tp}(A_2)\mathrm{tp}(\mathbb
Z)+\mathrm{tp}(A_3).
$$
This equality and ~(\ref{y6}) imply~(\ref{y2}). It is clear that we
can take $A_i\subseteq D_1$, $i=1,2,3$. Statement (i) follows.
\end{proof}
\begin{corollary}\label{c9}
Let $(X,d_X)\in \textbf{SM}$. If not all weak similarities
$\Phi:X\to X$ are isometries, then there are  $A_1$, $A_2$,
$A_3\subseteq D(X)$, $A_2\neq \varnothing$, such that
$$
\mathrm{tp}(D(X))=\mathrm{tp} A_1+\mathrm{tp} A_2\mathrm{tp} \mathbb
Z+\mathrm{tp} A_3.
$$
\end{corollary}

\section{Weak similarities, similarities and geodesics}

  Recall that a function $f:\mathbb{R}^+\to\mathbb{R}^+$ is
  subadditive if the inequality
  \begin{equation}\label{eq4.1}
  f(x+y)\leqslant f(x)+f(y)
  \end{equation}
holds for all $x,y \in \mathbb{R}^+$.

\begin{definition}\label{def4.1}
(\cite{DPK}) Let $A$ be a subset of $\mathbb{R}^+$. A function
$f:A\to\mathbb{R}^+$ is subadditive in the generalized sense if the
implication
  \begin{equation}\label{eq4.2}
(x\leqslant\sum\limits_{i=1}^mx_i)\Rightarrow (f(x)\leqslant
\sum\limits_{i=1}^mf(x_i))
  \end{equation}
holds for all $x, x_1,..,x_m\in A$ and every positive integer number
$m\geqslant 1$.
\end{definition}
\begin{remark}\label{rem4.2}
If $f:\mathbb{R}^+\to\mathbb{R}^+$ is increasing, then
~(\ref{eq4.1}) holds for all $x,y\in \mathbb{R}^+$ if and only if
~(\ref{eq4.2}) is true for all $x, x_1,...,x_m \in \mathbb{R}^+$
with $m\geqslant 2$. Thus Definition~\ref{def4.1} is equivalent to
the usual definition  of subadditivity if $A=\mathbb{R}^+$ and $f$
is increasing.
\end{remark}
\begin{lemma}\label{lem4.3}(~\cite{DPK})
Let $A$ be a nonempty subset of  $\mathbb{R}^+$. The following
conditions  are equivalent for every function $f:A\to\mathbb{R}^+$.
\begin{itemize}
  \item [(i)] The function $f$ is subadditive in the generalized sense.
  \item [(ii)] There is an increasing  and subadditive function $\Psi:\mathbb{R}^+\to\mathbb{R}^+$
such that $f$ is the restriction of $\Psi$ on $A$.
\end{itemize}
\end{lemma}
\begin{theorem}\label{th4.4}
Let $(X,d_X)$, $(Y,d_Y)\in \textbf{SM}$ and let $\Phi:X\to Y$ be a
weak similarity with the scaling function $f:D(Y)\to D(X)$. If
$f^{-1}$ and $f$ are subadditive in the generalized sense and $0\in
\mathrm{ac}(D(X)) \cap \mathrm{ac}(D(Y))$, then $\Phi$ is a
similarity.
\end{theorem}
Before proving the theorem we recall the definition of the lower
right Dini derivative. Let a real valued function $f$ be defined on
a set $A\subseteq \mathbb R$ and let $x_0\in A$. Suppose that
$x_0\in \mathrm{ac}(A\cap (x_0,\infty))$. The lower right Dini
derivative $D_+$ of $f$ at $x_0$  over set $A$ is defined by
$$
D_+f(x_0):=\liminf\limits_{\substack {x\to x_0\\
x\in A\cap (x_0,\infty)}}\frac{f(x)-f(x_0)}{x-x_0}.
$$
Analogously, the upper right Dini derivative of $f$ at $x_0$ over
set $A$ is defined as
$$
D^+f(x_0):=\limsup\limits_{\substack {x\to x_0\\
x\in A\cap (x_0,\infty)}}\frac{f(x)-f(x_0)}{x-x_0}.
$$
\begin{lemma}\label{lem4.5}
Let $A,B\subseteq \mathbb R$ and $f:A\to B$ be strictly increasing
and surjective and let $x_0\in A$, $y_0:=f(x_0)$. If $x_0\in
\mathrm{ac}(A\cap(x_0,\infty))$  and $y_0\in
\mathrm{ac}(B\cap(y_0,\infty))$, then the equality
\begin{equation}\label{eq4.3}
D^+f^{-1}(y_0)=\frac{1}{D_+f(x_0)}
\end{equation}
holds, where $f^{-1}$ is the inverse function for $f$ and
$D^+f^{-1}(y_0):=\{_0^{\infty}$ if $D_+f(x_0)=\{_{\infty}^0$.
\end{lemma}
\begin{proof}
It follows from the definition of Dini derivatives that
$$
D^+f^{-1}(x_0)=\limsup\limits_{\substack {y\to y_0\\
y\in B\cap (y_0,\infty)}}\frac{f^{-1}(y)-f^{-1}(y_0)}{y-y_0}
$$
$$
=\left(\liminf\limits_{\substack {y\to y_0\\
y\in B\cap
(y_0,\infty)}}\frac{y-y_0}{f^{-1}(y)-f^{-1}(y_0)}\right)^{-1}.
$$
The conditions
$$
x_0\in \mathrm{ac}(A\cap(x_0,\infty)) \, \text{ and } \, y_0\in
\mathrm{ac}(B\cap(y_0,\infty))
$$
imply that $f$ is right continuous at $x_0$ and $f^{-1}$ is right
continuous at $y_0$. Consequently we have
$$
\liminf\limits_{\substack {y\to y_0\\
y\in B\cap (y_0,\infty)}}\frac{y-y_0}{f^{-1}(y)-f^{-1}(y_0)}=\liminf\limits_{\substack {x\to x_0\\
x\in A\cap (x_0,\infty)}}\frac{f(x)-f(x_0)}{x-x_0}.
$$
Equality~(\ref{eq4.3}) follows.
\end{proof}

\begin{lemma}\label{lem4.6}
Let $A\subseteq \mathbb R^+$ and let $f:A\to \mathbb R^+$ be
subadditive in the generalized sense. Suppose that $0\in
A\cap\mathrm{ac}(A)$ and $f(0)=0$, then the inequality
\begin{equation}\label{eq4.4}
f(x)\leqslant D_+f(0)x
\end{equation}
holds for every $x\in A$.
\end{lemma}
\begin{proof}
Inequality~(\ref{eq4.4}) is trivial if $f(x)\equiv 0$ or
$D_+f(0)=+\infty$. Hence, without loss of generality, we can assume
\begin{equation}\label{eq4.5}
f(x)\not\equiv 0 \, \text{ and }\, D_+f(0)\neq +\infty.
\end{equation}
Since $f$ is subadditive in the generalized sense, ~(\ref{eq4.5})
implies that
\begin{equation}\label{eq4.6}
0 \leqslant  D_+f(0)<\infty
\end{equation}
and that the equivalence
\begin{equation}\label{eq4.7}
(f(x)=0)\Leftrightarrow (x=0)
\end{equation}
holds for every $x\in A$.

First consider the case when $A=\mathbb R^+$. Then, as has been
noted in Remark~\ref{rem4.2}, $f$ is increasing and subadditive.
Every increasing subadditive function satisfying ~(\ref{eq4.7}) is
metric preserving (see Theorem 4.1 in \cite{DM1}), i.e., $f\circ d$
is a metric for every metric space $(X,d)$. As has been shown in
Lemma 3.10~\cite{DM1},  a metric preserving function $f:\mathbb
R^+\to \mathbb R^+$ is Lipschitz if and only if $D_+ f(0)<\infty$.
Moreover, if this inequality holds, then $D_+f(0)$ is the Lipschitz
constant of $f$. Thus, if $A=\mathbb R^+$, then ~(\ref{eq4.4})
holds.

Suppose now that $A\neq \mathbb  R^+$. By Lemma~\ref{lem4.3} there
is an increasing subadditive function $\Psi:\mathbb R^+\to \mathbb
R^+$ such that
\begin{equation}\label{eq4.8}
\Psi(x)=f(x)
\end{equation}
for every $x\in A$. Since $A\subseteq \mathbb R^+$, the last
equality implies the inequality
\begin{equation}\label{eq4.9}
D_+\Psi(0)\leqslant D_+f(0).
\end{equation}
Furthermore we have also
$$
(\Psi(x)=0)\Leftrightarrow (x=0)
$$
for every $x\in \mathbb R^+$, because $\Psi(0)=f(0)=0$ and $\Psi$ is
increasing and subadditive. Hence, as has been shown above, the
inequality $\Psi(x)\leqslant D_+\Psi(0)x$ holds for every $x\in
\mathbb R^+$. The last inequality, ~(\ref{eq4.8}) and ~(\ref{eq4.9})
imply ~(\ref{eq4.4}) for every $x\in A$.
\end{proof}
\begin{proof}[Proof of theorem~\ref{th4.4}]
Suppose that $f$ and $f^{-1}$ are subadditive in the generalized
sense and
$$
a\in \mathrm{ac}(D(X))\cap \mathrm{ac}(D(y)).
$$
Using Lemma~\ref{lem4.6} we obtain the inequality
\begin{equation}\label{eq4.10}
f(y)\leqslant D_+f(0)y
\end{equation}
for every $y\in D(Y)$ and the inequality
\begin{equation}\label{eq4.11}
f^{-1}(x)\leqslant D_+f^{-1}(0)x
\end{equation}
for every $x\in D(X)$. Since $D_+f^{-1}(0)\leqslant D^+f^{-1}(0)$,
inequality~(\ref{eq4.11}) implies
\begin{equation}\label{eq4.12}
f^{-1}(x)\leqslant D^+f^{-1}(0)x.
\end{equation}
Note also that the double inequality
\begin{equation}\label{eq4.13}
0<D_+f(0)<\infty
\end{equation}
holds. Indeed, if $D_+f(0)=0$, then this equality, the inequality
$f(y)\geqslant 0$ and~(\ref{eq4.10}) imply  $f(y)\equiv 0$, contrary
to bijectivity of $f$. If $D_+f(0)=+\infty$, then Lemma~\ref{lem4.5}
gives $D^+f^{-1}(0)=0$. This equality and~(\ref{eq4.12}) imply
$f^{-1}(x)\equiv 0$, contrary to bijectivity of $f^{-1}$.

By Lemma~\ref{lem4.5} we have $D^+f^{-1}(0)=\frac{1}{D_+f(0)}$.
Substituting this equality in~(\ref{eq4.12}) we obtain
$$
f^{-1}(x)\leqslant \frac{1}{D_+f(0)}x
$$
for $x\in D(X)$ or, in the equivalent form,
\begin{equation}\label{eq4.14}
y\leqslant \frac{1}{D_+f(0)}f(y)
\end{equation}
for $y\in D(Y)$. Inequalities~(\ref{eq4.10}), ~(\ref{eq4.13}) and
~(\ref{eq4.14}) give the equality $f(y)=D_+f(0)y$ for every $y\in
D(Y)$. Now Lemma~\ref{lem3.2} implies that the weak similarity
$\Phi:X\to Y$ is a similarity, as required.
\end{proof}

The geodesic spaces are an important example of spaces for which
every weak similarity is a similarity.

We recall the definition of geodesics. Let $(X,d)$ be a metric
space. A \emph{geodesic path} in $X$ is a path $\gamma:[a,b]\to X$,
$-\infty<a<b<\infty$, such that
$d(\gamma(t_1),\gamma(t_2))=|t_2-t_1|$ for all $t_1, t_2 \in [a,b]$.
If $\gamma(a)=x$ and $\gamma(b)=y$, then we say that $\gamma$ joins
the points $x$ and $y$. A metric space $(X,d)$ is geodesic if for
every two distinct points $x_1, x_2\in X$ there is a geodesic path
in $X$ joining them (see, for example, \cite[p. 51, p 58]{Pa}).
\begin{theorem}\label{th4.7}
Let $(X,d_X)$ and $(Y,d_Y)$ be geodesic metric spaces and let
$\Phi:X\to Y$ be a weak similarity. Then $\Phi$ is a similarity and,
if $X$ and $Y$ are bounded with $0<\diam X \wedge \diam Y$, then the
ratio $r(\Phi)$ equals $\frac{\diam Y}{\diam X}$.
\end{theorem}
\begin{proof}
Let $f:D(Y)\to D(X)$ be the scaling function corresponding $\Phi$.
By Theorem~\ref{th4.4} to prove that $\Phi$ is a similarity it is
sufficient to show that $f$ and $f^{-1}$ are subadditive in the
generalized sense. By definition $f$ is subadditive in the
generalized sense if the inequality
\begin{equation}\label{eq4.15}
f(t)\leqslant \sum\limits_{i=1}^mf(t_i)
\end{equation}
holds for $t, t_1, t_2,...,t_m\in D(Y)$ whenever
\begin{equation}\label{eq4.16}
t\leqslant \sum\limits_{i=1}^mt_i,\quad m\in \mathbb N.
\end{equation}
Let $t, t_1, t_2,...,t_m\in D(Y)$, $t>0$, and let~(\ref{eq4.16})
hold. Write
$$
\alpha:=\frac{t}{\sum\limits_{i=1}^{m}t_i}.
$$
It is clear that $\alpha \leqslant 1$ and
$t=\sum\limits_{i=1}^{m}\alpha t_i$. Let $a, b\in Y$ with
$d_Y(a,b)=t$ and let $\gamma:[0,t]\to Y$ be a geodesic path joining
$a$ and $b$. Let us define points $y_0, y_1,....,y_m\in Y$ and $x_0,
x_1,...,x_m\in X$ as
$$
y_0:=\gamma(0)=a,\, y_1:=\gamma(\alpha t_1),\, y_2:=\gamma(\alpha
t_1+\alpha t_2),...,$$
$$
y_m:=\gamma (\sum\limits_{i=1}^m\alpha
t_i)=\gamma(t)=b
$$
and write $x_i:=\Phi^{-1}(y_i)$, $i=0,...,m$. The triangle
inequality implies
$$
d_X(x_0,x_m)\leqslant \sum\limits_{i=0}^{m-1}d_X(x_i,x_{i+1}).
$$
Since  $\Phi$  is a weak similarity with the scaling function $f$,
the last inequality can be written as
$$
(f\circ d_Y)(y_0,y_m)\leqslant \sum\limits_{i=0}^{m-1}(f\circ
d_Y)(y_i,y_{i+1})
$$
or as
$$
(f\circ d_Y)(a,b)\leqslant (f\circ d_Y)(\gamma(0),\gamma(\alpha
t_1))$$
$$+(f\circ d_Y)(\gamma(\alpha t_1), \gamma (\alpha t_1+\alpha
t_2))+...+(f\circ d_Y)(\gamma(\sum\limits_{i=1}^{m-1}\alpha t_i),
\gamma(\sum\limits_{i=1}^{m}\alpha t_i)).
$$
Since $\gamma$ is a geodesic joining  $a$ and $b$ we have from the
previous inequality that
$$
f(t)\leqslant f(\alpha t_1)+f(\alpha t_2)+...+f(\alpha t_m).
$$
The last inequality, the increase of $f$ and inequality $\alpha
\leqslant 1$ implies ~(\ref{eq4.15}). Consequently $f$ is
subadditive in the generalized sense. The generalized subadditivity
of $f^{-1}$ can be proved similarly. It still remains to note that
the equality $r(\Phi) \diam X =\diam Y$ holds for every similarity
$\Phi:X\to Y$.
\end{proof}
\begin{corollary}\label{cor4.8}
Let $(X,d_X)$ and $(Y,d_Y)$ be bounded geodesic spaces. If $\diam
X=\diam Y$ and $X\we Y$, then $X$ and $Y$ are isometric.
\end{corollary}

To construct an example of compact weak equivalent metric spaces
which are not isometric but have the same diameter we shall use the
\emph{snow-flake transformation} $d \mapsto d^p$, $p\in (0,1)$. It
is well known $d^p$ is a metric for every metric $d$ and $p\in
(0,1]$ (see, for example,~\cite[p. 97]{BS}).
\begin{example}\label{ex3.11}
Let $X=[0,1]$, $d_X(x,y)=\sqrt{|x-y|}$ and  $Y=[0,1]$,
$d_Y(x,y)=|x-y|$. It is clear that $D(Y)=D(X)=[0,1]$. The spaces
$(X,d_X)$ and $(Y,d_Y)$ are weak equivalent with the  realization
$(f,\Phi)$ where $f(x)=\sqrt{x}$ and $\Phi(x)=x$ for every $x\in
[0,1]$. It is easy to see that $(X,d_X)$ and $(Y,d_Y)$ are compact
and $\diam X= \diam Y=1$.

The space $(Y,d_Y)$ is geodesic. Since there are no rectifiable
paths joining $0$ and $1$ in $X$, the space $(X,d_X)$ is not
geodesic. Hence $(X,d_X)$ and $(Y,d_Y)$ are not isometric.
\end{example}

{\bf Oleksiy Dovgoshey}

Institute of Applied Mathematics and Mechanics of NASU, R. Luxemburg str. 74, Donetsk 83114, Ukraine

{\bf E-mail: } aleksdov@mail.ru
\bigskip

{\bf Evgeniy Petrov}

Institute of Applied Mathematics and Mechanics of NASU, R. Luxemburg str. 74, Donetsk 83114, Ukraine

{\bf E-mail: } eugeniy.petrov@gmail.com

\end{document}